\documentclass[11pt]{article}
\usepackage{hyperref}
\normalsize

\usepackage{amsfonts, amsthm, amsgen}

\usepackage[latin1]{inputenc}

\usepackage{authblk}
\usepackage{amssymb,latexsym,amsmath}
\usepackage{amsthm}

\newtheorem{theorem}{Theorem}[section]
\newtheorem{proposition}[theorem]{Proposition}
\newtheorem{lemma}[theorem]{Lemma}

\theoremstyle{definition}
\newtheorem{definition}[theorem]{Definition}
\newtheorem{example}[theorem]{Example}

\newtheorem{remark}[theorem]{Remark}
\newtheorem{remarks}[theorem]{Remarks}

\newcommand{\F}{\mathbb{F}}
\newcommand{\N}{\mathbb{N}}
\newcommand{\Z}{\mathbb{Z}}

\newcommand{\wt}{\mathrm{wt}}

\newcommand{\mC}{\mathcal{C}}

\newcommand{\Fq}{\mathbb{F}_{q}}
\newcommand{\Fqn}{\mathbb{F}_{q}^{n}}
\newcommand{\Fqm}{\mathbb{F}_{q^{m}}}

\newcommand{\supp}{\mathrm{supp}}

\newcommand{\defect}{\mathrm{def}}

\newcommand{\Fqmn}{\mathbb{F}_{q^{m}}^{n}}

\newcommand{\dist}{\mbox{\rm d}}
\newcommand{\AMRD}{\mbox{\rm AMRD}}
\newcommand{\MRD}{\mbox{\rm MRD}}

\DeclareMathOperator{\rank}{rank}

\title{\bf On $q$-analog Steiner systems of rank metric codes}

\author{Francisco Arias$^1$,  Javier de la Cruz$^{1,2}\footnote{
This work was done while J. de la Cruz was at the University of
Zurich supported by the Swiss Confederation through the
Swiss Government Excellence Scholarship no. 2016.0873.
 The autor was partially supported by COLCIENCIAS through project no. 121571250178.}$,
Joachim Rosenthal$^2$\footnote{J. Rosenthal
was supported in part by the Swiss National Science Foundation under grant no. 169510.}
and Wolfgang Willems$^{1,3}$ \\

{\footnotesize ${}^1$}Universidad del Norte, Barranquilla, Colombia \\
{\footnotesize $^2$}University of Zurich,   Switzerland\\
{\footnotesize $^3$}Otto-von-Guericke Universit\"at, Magdeburg, Germany}

\begin{document}

\maketitle

\begin{abstract}
  In this paper we prove that rank metric codes with special
  properties imply the existence of $q$-analogs of suitable designs.
  More precisely, we show that the minimum weight vectors of a
  $[2d,d,d]$ dually almost MRD code $C\leq \Fqmn$ which has no code
  words of rank weight $d+1$ form a $q$-analog Steiner system
  $S_q(d-1,d,2d)$. In particular, $d+1$ must be a prime.
\end{abstract}

\section{Introduction}
The interest in $q$-analogs of codes and designs has been increased
over the last years due to their applications in random network
coding. One of the most challenging problems is the existence of
$q$-analogs of Steiner systems, in particular of the Fano plane.

The paper is structured as follows. In Section \ref{prelim} we collect
some facts on rank metric codes, in particular on generalized rank
weights.  Section \ref{Gauss} deals with Gaussian binomial
coefficients and cyclotomic polynomials.  In Section \ref{support} we
analyze the supports of the minimum weight vectors of a rank metric
code. Section \ref{steiner} deals with a relationship between rank
metric codes and $q$-analog designs. We prove that the minimum weight
vectors of a $[2d,d,d]$ dually almost MRD code $C\leq \Fqmn$ which has
no code words of rank weight $d+1$ hold a $S_q(d-1,d,2d)$ Steiner
system. In particular $d+1$ must be a prime. Note that apart from
trivial examples only $S_2(2,3,13)$ is known to exist \cite{Braun}.

\section{Preliminaries} \label{prelim}

In this paper we study $\F_{q^m}$-linear codes $ C \leq \F_{q^m}^n$
endowed with the rank metric distance.  To be more precise, note that
the field $\F_{q^m}$ may be viewed as an $m$-dimensional vector space
over $\F_q$. The {\it rank weight}, or briefly the {\it weight} of a
vector $v=(v_1,\dots,v_n) \in \F_{q^m}^n$ is defined as the maximum
number of coordinates in $v$ that are linearly independent over
$\F_q$, i.e., $\wt(v)= \dim_{\F_q}\langle v_1,\dots,v_n \rangle$. For
$v,u \in \F_{q^m}^n$ the rank metric distance is then given by
$ \dist(v,u) =\wt(u-v) =\rank(v-u). $

An $\F_{q^m}$-linear subspace $C \leq \F_{q^m}^n$ of dimension $k$
endowed with this metric is called an $[n,k]$ {\it $\Fqm$-linear rank
  metric} code.  As usual the minimum distance of $C \neq \{ 0\}$ is
defined by $$\dist = \dist(C)=\min \{ \wt(c) \mid 0\not=c \in C\}.$$
By $A_i(C)$ we always denote the code words of $C$ of weight $i$.
Finally, we use the notation $C^\perp$ for the orthogonal of $C$ which
is taken with respect to the standard inner product of $\Fqmn.$


Throughout the paper we always assume that $C \leq \F_{q^m}^n$ is an
$\F_{q^m}$-linear rank metric code with minimum distance $d$.
Furthermore we assume that $C$ is not trivial, i.e.,
$0 \not= C \not= \F_{q^m}^n$ and $n \leq m$. Thus, if $\dim C = k$,
then the last condition implies the Singleton bound
$$
d \leq n-k +1.
$$
$C$ is called a {\it maximum rank distance code}, shortly an $\MRD$
code, if the bound is achieved.  Delsarte \cite{Delsarte} and
independently Gabidulin \cite{Gabidulin} proved the existence of such
codes for all $q,m, n$ and dimension $1\leq k \leq n$ (here $n \leq m$
is not necessary). Given the parameters $q,m, n, k$, the code
$C\leq \Fqmn$ these authors describe has a particular construction
through a generator matrix $M_k(v)$ and the resulting code is usually
called a Gabidulin code. Recently other new constructions of MRD codes
have been found which are not equivalent to Gabidulin codes
(\cite{Willems-DelaCruz-K-W,sheekey}). Somehow surprisingly, over the
algebraic closure, the set of MRD codes forms a generic set inside the
Grassmann variety of all $k$-dimensional linear subspaces of
$\F_{q^m}^n$~\cite{ne17}. In particular over some large finite field
there exist large numbers of MRD codes and lower bounds on these
cardinalities can be found in~\cite{ne17}.

In analogy to the Singleton defect for classical codes as given in
\cite{debo96,Faldum-Willems}, we have the following definition for the
defect of rank metric codes \cite{DC-G-R-L}.

\begin{definition} \label{rkdef-Delsarte} The {\it rank defect},
  briefly the {\it defect}, of an $\Fqm$-linear $[n,k, d]$ rank metric
  code $C \leq \Fqmn$ is defined by $\defect(C)=n-k+1-d$.
\end{definition}

Note that $\defect(C)=0$ if and only if $C$ is an MRD code. Other
interesting codes which are coming close to MRD codes, are the
so-called {\it dually almost} MRD codes or simply {\it dually \AMRD}
codes \cite{DC}. More precisely, we say that a $\F_{q^m}$-linear rank
metric code $C$ is {dually AMRD} if
$\defect(C)=\defect(C^\perp)=1$. Dually AMRD codes are subject of the
main results in the last section of this paper. These codes can be
viewed as a $q$-analogon of a classical almost-MDS (AMDS) code and as
in the classical situation these codes induce again some $q$-Steiner
system.

Let $b_1, \ldots, b_m$ be a basis $B$ of $\F_{q^m}$ over $\F_q$.  For
$v =(v_1, \ldots,v_n) \in \F_{q^m}^n$ we write
$$
v_i= \sum_{j=1}^m \alpha_{ji}b_j
$$
and put $M_B(v) = (\alpha_{ji}) \in (\F_q)^{m \times n}.$ As mentioned
in (\cite{Relinde}, Section 2), the $K$-linear row space of $M_B(v)$
is independent of the chosen basis $B$.

In order to define generalized rank weights we need the following
notations \cite{H-M-R, Relinde}.

\begin{definition}
  For $v=(v_1, \ldots,v_n) \in \F_{q^m}^n$ and an $\F_{q^m}$-linear
  subspace $V$ of $\F_{q^m}^n$ we define
  \begin{itemize}
  \item[\rm a)] $\supp(v)$ as the $\F_q$-linear row space of $M_B(v)$.
  \item[\rm b)] $\supp(V) =\langle \supp(v) \mid v \in V \rangle$ as
    an $\F_q$-vector space.
  \item[\rm c)] $\wt(V) = \dim \supp(V).$
  \item[\rm d)] $ V^\star = \sum_{i=0}^{m-1} V^{q^i}.$
  \end{itemize}
\end{definition}

In the literature there are different definitions for generalized rank
weights (see \cite{Ogg},\cite{Kurihara-Matsumoto-Uyematsu},
\cite{Ducoat}, \cite{Relinde}).  All of them define the same
numbers. For our purpose the definition given in \cite{Relinde} seems
to be the most appropriate.

\begin{definition}
  The {\it $r$-th generalized rank weight } $\dist_r$ of a rank metric
  code $C \leq \F_{q^m}^n$ is defined by
$$
\dist_r(C) = \min_{D \leq C \atop \dim D =r} \wt(D).
$$
\end{definition}

Combining results of \cite{Kurihara-Matsumoto-Uyematsu},\cite{Ducoat}
and \cite{Relinde} we obtain the rank metric analog of Wei's result
\cite{Wei} on generalized Hamming weights.

\begin{theorem} \label{gen-weights} If $C$ is an $\F_{q^m}$-linear
  rank metric code in $\F_{q^m}^n$ of dimension $k$ and minimum
  distance $d$, then
$$
\dist(C) = \dist_1(C) < \dist_2(C) < \ldots < \dist_k(C).
$$
\end{theorem}
\begin{proof} We have
$$
\begin{array}{rcll}
  \dist_r(C)  & = &   \min\limits_{D \leq C \atop \dim D =r}  \wt(D)& \\
              & = &  \min\limits_{D \leq C \atop \dim D =r} \dim D^\star & (\cite{Relinde}, \text{Corollary 4.4}) \\
              & = & \min\limits_{D \leq C \atop \dim D =r} \max_{d \in D^\star} \wt(d) & (\cite{Relinde},\text{Theorem 5.8})\\
              & = & \min\limits_{V=V^\star \atop \dim (C \cap V) \geq r} \dim V & (\cite{Ducoat}, \text{Proposition II.1})\\
              & = & {\cal M}_r(C). & (\text{Definition 5 in \cite{Kurihara-Matsumoto-Uyematsu}})
\end{array}
$$
By (\cite{Kurihara-Matsumoto-Uyematsu}, Lemma 9) we get
$$
{\cal M}_1(C) < \ldots < {\cal M}_k(C),
$$ and the proof is complete since obviously $\dist(C) = \dist_1(C)$.
\end{proof}

\section{Gaussian binomial coefficients and cyclotomic
  poly\-nomials} \label{Gauss}

The results of this section are known but crucial for the rest of the
paper. Since they are hard to find in the literature we will state
them with proofs for the reader's convenience.

\begin{definition}
  Let $q$ be a prime power and let $a$ and $b$ be non-negative
  integers.  The $q$-ary Gaussian binomial coefficient of $a$ over $b$
  is defined by
 $$
 {{a \brack b}_q} = \left\{
   \begin{array}{cl}
     \frac{(q^a-1)(q^{a-1}-1)\ldots(q^{a-b+1}-1)}{(q^b-1)(q^{b-1}-1)\ldots(q-1)}
     & \mbox{ if } b \leq a \\ 0 & \mbox{ if } b > a
   \end{array}
 \right.
$$
\end{definition}

Throughout the paper we freely use the symmetry of the Gaussian
binomial coefficients; i.e., ${{a \brack b}_q}= {{a \brack a-b}_q}$
for $ b \leq a$.

Furthermore ${{a \brack b}_q}$ can be expressed by suitable
$\Phi_n(q)$ where $\Phi_n(x)$ denotes the $n$-th cyclotomic polynomial
defined by
$$
\Phi_n(x) = \prod_{1 \leq i \leq n \atop \gcd(j,n)=1} (x - \zeta_n^j)
$$
where $\zeta_n$ is a primitive complex $n$-th root of unity. Recall
that $\Phi_n(x)$ is an irreducible polynomial in $\Z[x]$. For
$n \in \N$ we put $[n] =\{1,2, \ldots, n\}$.

\begin{proposition} \label{factorization} For $b < a$ we have
$$
{{a \brack b}_q} = \prod_{j \in J_{a,b}} \Phi_j(q)
$$
where
$J_{a,b} = \{ j \in [a] \mid ((a-b) \bmod j) + (b \bmod j) \geq j \}$.
\end{proposition}
\begin{proof} By (\cite{Chen}, Lemma 1), we have
$$
{{a \brack b}_q} = \prod_{j=1}^a \Phi_j(q)^{ \lfloor
  \frac{a}{j}\rfloor - \lfloor \frac{b}{j}\rfloor - \lfloor
  \frac{a-b}{j}\rfloor}.
$$
Furthermore, since
$$
0 \leq \lfloor \frac{a}{j}\rfloor - \lfloor \frac{b}{j}\rfloor -
\lfloor \frac{a-b}{j}\rfloor \leq 1
$$
we obtain
$$
{{a \brack b}_q} = \prod_{j \in J} \Phi_j(q)
$$
where
$J = \{ j \in [a] \mid \lfloor \frac{a}{j}\rfloor = \lfloor
\frac{b}{j}\rfloor + \lfloor \frac{a-b}{j}\rfloor + 1\}$.
Thus we need to show that $J = J_{a,b}$. If we write
$a = \lfloor \frac{a}{j} \rfloor j + r_a $ with $ 0 \leq r_a < j$ and
similarly $b$ and $a-b$ we get
$$
a = \Big(\Big\lfloor \frac{b}{j} \Big\rfloor + \Big\lfloor
\frac{a-b}{j} \Big\rfloor\Big)j + r_b + r_{a-b}.
$$
Thus $ j \in J$ if and only if
$$
r_b + r_{a-b} -j = r_a \geq 0
$$
if and only if
$$
r_b + r_{a-b} \geq j .
$$
The last condition says nothing else than
$$
(b \bmod j) + ((a-b) \bmod j) \geq j.
$$
\end{proof}

\begin{lemma} \label{$J_{a,b}$} Let $a,d \in \N$.
  If $p$ is a prime with $p\mid d+1$ and $p \mid c $, then
  $c \not\in J_{d+p, p-1}$.
\end{lemma}
\begin{proof}
  Write $d+1 = x c +r$ with $x \in \N$ and $0 \leq r < c$. Since
  $ p \mid d+1$ and $p\mid c$ we have $ p \mid r$. Suppose that
  $c\in J_{d+p,p-1}$.  Thus
$$ ((d+1) \bmod c) + ((p-1) \bmod c) \geq c.$$
This implies that $ r + (p-1) \geq c$, hence $ c > r \geq c -p+1$.
Thus we obtain $ r = c -p+i$ where $i\in \{1, \ldots, p-1$, which is a
contradiction since $ p\mid r$ and $p\mid c$.
\end{proof}


\begin{lemma} \label{ge} Let $p$ be a prime and $c \in \N$. If
  $\gcd(\Phi_p(q),\Phi_c(q)) > 1$, then $p\mid c$.
\end{lemma}
\begin{proof}
  The assumption $\gcd(\Phi_p(q),\Phi_c(q)) > 1$ implies that
  $\gcd(q^p-1,q^c-1) > 1$. From finite field theory we know that
$$
\gcd(q^p-1,q^c-1) = q^{\gcd(p,c)}-1.
$$
Thus, if $p \nmid c$, then $\gcd(q^p-1,q^c-1) = q-1 =
\Phi_1(q)$. Since
$$
q^p-1 = \Phi_1(q) \Phi_p(q)
$$
and
$$
q^c-1 = \Phi_1(q)\prod_{1\not=t \mid c} \Phi_t
$$
we obtain $\gcd(\Phi_p(q),\Phi_c(q))=1$, a contradiction.
\end{proof}

\section{Supports of the minimum weight vectors} \label{support}

From paper \cite{Relinde} we know the following facts.

\begin{lemma} \label{prop-supp} Let $ C \leq \F_{q^m}^n$ be an
  $\F_{q^m}$-linear rank metric code.
  \begin{itemize}
  \item[\rm a)] If $u=\alpha v$ for some $\alpha \in \Fqm$, then
    $\supp(v)=\supp(u)$.
  \item[\rm b] If $v_1, \ldots, v_k \in \Fqmn$ generate $C$,
    then $$\supp(C)=\sum_{i=1}^k \supp(v_i).$$
  \item[\rm c)] There exists an element $c\in C$ such
    that $$\supp(c)=\supp(C).$$
  \item[\rm d)] For $u,v \in \F_{q^m}^n$ there exist
    $\alpha, \beta \in \Fqm$ such that
    $\supp(\alpha v+\beta u)=\supp(v) +\supp(u)$.
  \end{itemize}
\end{lemma}

\begin{proof} a) and b) are part of Proposition 2.3 of \cite{Relinde}.
  c) is Proposition 3.6 and d) Proposition 3.9 of the same paper.
\end{proof}

\begin{definition}
  For an $\Fqm$-linear rank metric code $C \leq \Fqmn$ of dimension
  $k$ and minimum distance $d$ we put
$$ D_i(C) = \{ \supp(c) \mid c \in C, \  \wt(c) = i\}$$
for $i=0,d, \ldots, n-k+1$.
\end{definition}

\begin{lemma}\label{number-block} Let $C \leq \Fqmn$ be an
  $\Fqm$-linear rank metric code with minimum distance $d$.
  \begin{itemize}
  \item[\rm a)] Let $v,u \in C$ and $\wt(v)=\wt(u)=d$. Then
    $\supp(v)=\supp(u)$ if and only if there exists
    $\alpha \in \Fqm^\star$ such that $u=\alpha v$.
  \item[\rm b)] $|D_d(C)|=\frac{A_d(C)}{q^m-1}$.
  \end{itemize}
\end{lemma}
\begin{proof} a) One direction follows by Lemma \ref{prop-supp} a).
  Suppose $\supp(v)=\supp(u)$ and $v, u$ linearly independent over
  $\Fqm$. Let $W=\langle v,u\rangle$ as a vector space over $\Fqm$.
  By Lemma \ref{prop-supp} b), we get
  $\supp(W)=\supp(v)+\supp(u)=\supp(v).$
  Therefore $$\wt(W)=\dim_{\Fq}(\supp(W))=\dim_{\Fq}(\supp(v))=d.$$
  Thus, according to the definition of generalized rank weights we
  obtain
$$
d_2(C)=\min \{ \wt_R(S) \mid S \leq C \;\textrm{and}\; \dim_{\Fqm}S=2\}=d,
$$
which contradicts Theorem \ref{gen-weights}.\\
b) This immediately follows from part a).
\end{proof}

\section {q-analog Steiner systems and rank metric
  codes} \label{steiner}

Maximum distance separable (MDS) codes are $[n,k,d]$ linear codes
$C\leq \Fqn$ which reach the Singleton bound $d=n-k+1$. Almost-MDS
(AMDS) codes were introduced by de Boer~\cite{debo96} and they are
characterized that their Singleton defect is one, i.e.  $d=n-k$.

In \cite{Faldum-Willems} it has been shown that the supporters of code
words of minimum weight of a $[2d,d,d]$ dually AMDS code ($d \geq 2$)
which has no code words of weight $d+1$ form the blocks of an
$S(d-1,d,2d)$ classical Steiner system and $d+1$ must be a prime.  For
instance, in this way the extended ternary Golay code leads to an
$S(5,6,12)$ Steiner system. In this section we prove the $q$-analog of
this result.

\begin{definition}
  Let $t \leq k \leq n$ be natural numbers. A $q$-Steiner system
  $S_q(t,k,n)$ is a set of $k$-dimensional subspaces of $\F_q^n$,
  called the blocks, such that every $t$-dimensional subspace of
  $\F_q^n$ is contained in exactly one block.
\end{definition}

Note that the number of blocks of an $S_q(t,k,n)$ Steiner system is
$\frac{{{n \brack t}_q}}{{{k \brack t}_q}}$.

\begin{lemma} \label{st-ab} A $S_q(t,k,n)$ Steiner system implies an
  $S_q(t-1,k-1,n-1)$ Steiner system if $t \geq 2$.
\end{lemma}
\begin{proof}
  This is one part of (\cite{KL}, Lemma 5).
\end{proof}

\begin{theorem}\label{design-2d-d-d}
  Let $C \leq \Fqm^{2d}$ be a $[2d,d,d]$ dually AMRD code with
  $d \geq 2$ and $A_{d+1}(C)=0$. Then the set $D_d(C)$ are the blocks
  of an $S_q(d-1,d,2d)$ Steiner system.
\end{theorem}
\begin{proof}
  (i) Let $ W \leq \F_q^{2d}$ be of dimension $d-1$. Suppose that $W$
  is contained in two different blocks, i.e., elements of $D_d(C)$.
  Hence $$ W \subseteq \supp(u) \cap \supp(v)$$ with
  $\supp(u), \supp(v) \in D_d(C)$. Since
  $ \dim \, (\supp(u) \cap \supp(v)) \leq d-1$ we obtain
$$
W = \supp(u) \cap \supp(v).
$$
Thus
$$
\dim \, (\supp(u)+ \supp(v)) =2d -(d-1) = d+1.
$$
By Lemma \ref{prop-supp} d) there are $\alpha, \beta \in \F_{q^m}$
such that
$$
\supp(u) + \supp(v) =\supp(\alpha u + \beta v).
$$
Thus $\alpha u + \beta v \in C$ has weight $d+1$, a contradiction.
This means that every $(d-1)$-dimensional subspace of $\F_q^{2d}$ is contained in at most one block. \\
(ii) According to Lemma \ref{number-block} b) we have
$|D_d(C)| = \frac{A_d(C)}{q^m -1}$. Since $A_{d+1}(C)=0$, Theorem 25
of \cite{DC-G-R-L} yields
$$
A_d(C) = \frac{ {{2d \brack d+1}_q}}{{{d \brack 1}_q}}(q^m-1) = \frac{
  {{2d \brack d-1}_q}}{{{d \brack d-1}_q}}(q^m-1),
$$
hence $|D_d(C)| = \frac{ {{2d \brack d-1}_q}}{{{d \brack d-1}_q}}$.
Since each block contains exactly ${{d \brack d-1}_q}$ subspaces of
dimension $(d-1)$ and every $(d-1)$-dimensional subspace is contained
in at most one block by (i), the blocks altogether contain
$$
|D_d(C)|{{d \brack d-1}_q} = {{2d \brack d-1}_q}
$$
subspaces of dimension $d-1$. As ${{2d \brack d-1}_q}$ is the number
of $(d-1)$-dimensional subspaces in a space of dimension $2d$, the
proof is complete.
\end{proof}

\begin{remark}
  Let $C \leq \Fqm^{2d}$ be a $[2d,d,d]$ dually AMRD code with
  $d \geq 2$ and $A_{d+1}(C)=0$. Then $C^\perp$ also leads to an
  $S_q(d-1,d,2d)$ Steiner system, since $C$ is formally self-dual, by
  (\cite{DC}, Lemma 4.11).
\end{remark}

\begin{example}
  Let $C$ be the $\F_{2^4}$-linear $[4,2,2]$ code with generator
  matrix $$\left(
    \begin{array}{ccll}
      0& 1 & \omega & 0 \\
      1& 0 & 0 & \omega \\
    \end{array}
  \right) $$
  where $\omega$ is a primitive third root of unity in
  $\F_{2^4}^\star$.  With {\sc Magma} \cite{Magma} we get
  $A_0(C)= A_0(C^\perp)=1$, $A_2(C) = A_2(C^\perp)=75$,
  $A_3(C) = A_3(C^\perp)=0$ and $A_4(C) = A_4(C^\perp)=180$. Thus $C$
  is a $[4,2,2]$ dually almost MRD code over $\F_{2^4}$. Consequently,
  by Theorem \ref{design-2d-d-d} the elements of $D_d(C)$ are the
  blocks of an $S_2(4,2,1)$ Steiner system. Note that this $2$-Steiner
  system is one of the trivial ones.
\end{example}

\begin{remarks} a) According to Theorem \ref{design-2d-d-d} a
  $[8,4,4]$ dually AMRD code over $\F_{2^8}$ with $A_5(C)=0$ implies
  the existence of a Steiner system $S_2(3,4,8)$.  Thus, by Lemma
  \ref{st-ab}, the existence of the code would imply the existence of
  an $S_2(2,3,7)$ Steiner system which is
  the $2$-analog of the Fano plane. \\
  b) By Theorem \ref{design-2d-d-d} and Lemma \ref{st-ab}, a
  $[2d,d,d]$ dually AMRD code over $\F_{q^m}$ with $d \geq 2$ and
  $A_{d+1}(C)=0$ implies an $S_q(1,2,d+2)$ Steiner system. It follows
  that $q^2-1 \mid q^{d+2} -1$.  Thus $d$ must be even.
\end{remarks}

\begin{theorem} \label{d+1} Let $C \leq \Fqm^{2d}$ be a $[2d,d,d]$
  dually AMRD code with $d \geq 2$ and $A_{d+1}(C)=0$. Then $d+1$ is a
  prime.
\end{theorem}
\begin{proof} Let $p$ be a prime with $p \mid d+1 \not= p$, hence
  $d+1 =px$ with $x \geq 2$.  By Theorem \ref{design-2d-d-d}, there
  exists a Steiner system $S_q(d-1,d,2d)$. Since $p-1 \leq d-1$ Lemma
  \ref{st-ab} implies the existence of an $S_q(p-1,p,d+p)$ Steiner
  system.  This Steiner system has exactly
  $$\frac{{{d+p \brack p-1}_q}}{{{p \brack p-1}_q}} = \frac{{{d+p
        \brack p-1}_q}}{{{p \brack 1}_q}} \in \N $$
  blocks.  According to Proposition \ref{factorization} we obtain
$$
\frac{{{d+p \brack p-1}_q}}{{{p \brack 1}_q}} = \frac{ \prod_{j \in
    J_{d+p,p-1}} \Phi_j(q)}{ \prod_{j \in J_{p,1} }\Phi_j(q) } =
\frac{ \prod_{j \in J_{d+p,p-1}} \Phi_j}{ \Phi_p(q) } \in \N.
$$
Thus exists a $c \in J_{d+p,p-1}$ such that
$1 < \gcd(\Phi_p(q),\Phi_c(q))$.  Lemma \ref{ge} implies that
$ p \mid c$ and according to Lemma \ref{$J_{a,b}$} we get
$ c \not\in J_{d+p,p-1}$, a contradiction. Thus $ d+1 = p$ and we are
done.
\end{proof}

\end{document}